\documentclass[10pt]{amsart}
\pdfoutput=1
\usepackage[utf8]{inputenc}
\usepackage[T1]{fontenc}
\usepackage{amssymb}
\usepackage[dvipsnames,svgnames,x11names,hyperref]{xcolor}
\usepackage{amsmath, amsfonts,amsthm}
\usepackage[colorlinks=true]{hyperref}
\usepackage[all]{xy}
\usepackage{xypic}
\numberwithin{equation}{subsection}
\theoremstyle{plain}
\newtheorem{theorem}[subsection]{Theorem}
\theoremstyle{definition}
\newtheorem{definition}[subsection]{Definition}

\newtheorem*{notation*}{Notation}
\theoremstyle{plain}
\newtheorem{proposition}[subsection]{Proposition}
\newtheorem{lemma}[subsection]{Lemma}
\theoremstyle{corollary}

\theoremstyle{definition}
\theoremstyle{remark}
\usepackage{geometry}
\fontfamily{ptm}
\selectfont
\usepackage{mathrsfs}
\usepackage{tikz-cd}
\usepackage[utf8]{inputenc}
\usepackage{amsmath}
\usepackage{amsfonts}
\usepackage{amssymb}
\usepackage{graphicx}
\usepackage[T1]{fontenc}
\usepackage{amsthm}
\usepackage{enumitem}
\usepackage{graphicx,nicefrac}
\usepackage[all]{xy}
\usepackage{mathtools}
\usepackage{xcolor}
\usepackage{tikz}
\usepackage{tikz-cd}
\usepackage{mathrsfs}
\tikzstyle{line}=[draw]

\title{New description of perverse sheaves on a disc}
\author{Krystian Olechowski}
\address{}
\email{}
\begin{document}
\hypersetup{linkcolor=Teal}
\hypersetup{citecolor=Teal}
\hypersetup{urlcolor=Teal}

\bibliographystyle{alpha}

\begin{abstract}
	 There is a connection between the category of perverse sheaves on a disc and different notions related to spherical functors. We introduce a category whose objects are analogous to 4-periodic semiorthogonal decompositions and prove that it is equivalent to the category of perverse sheaves on a disc stratified by the origin and its complement.
\end{abstract}

\maketitle

\section{Introduction}

A. Beilinson proved in \cite{Bei} that the category of perverse sheaves on a disc stratified by the origin and its complement (denoted by 
$\mathsf{Perv}(D)$) is equivalent to the category of pairs of vector spaces together with two composable linear transformations $u$ and $v$ between them satisfying the condition that $1 - uv$ is an isomorphism. We will denote this category as $\mathcal{A}_1$.

In particular, this description of $\mathsf{Perv}(D)$ gave rise to the connection between perverse sheaves and spherical functors. Spherical twists were first introduced by P. Seidel and R. Thomas in \cite{st}. The notion of a spherical functor was later scrutinized by R. Anno and T. Logvinenko in \cite{al}. To properly define the spherical functor, they used the language of DG categories. Roughly, the definition is as follows.

Let $\mathcal{A}, \mathcal{B}$ be DG categories and $s: \mathcal{D}(\mathcal{A}) \to \mathcal{D}(\mathcal{B})$ be a functor 
between their derived categories. Assume that $s$ has both adjoints, $l \dashv s \dashv r$. The adjunction unit and counit yield two distinguished triangles of functors:
\begin{equation} \label{sphericaltriangles}
\begin{aligned}
sr \to Id \to t, \\
f \to Id \to rs
\end{aligned}
\end{equation}
for some functors $f : \mathcal{D}(\mathcal{A}) \to \mathcal{D}(\mathcal{A})$ and $t : \mathcal{D}(\mathcal{B}) \to \mathcal{D}(\mathcal{B})$.
We say that $s$ is spherical if $f$ and $t$ are equivalences of categories.

By considering the triangles (\ref{sphericaltriangles}) in the Grothendieck group one can see that relations defining
spherical functors say precisely that $Id - sr$ and $Id - rs$ are equivalences of categories.

This description shows a direct relation between spherical functors and the category $\mathcal{A}_1$ - vector spaces are replaced by categories and linear transformations by functors. This relation first appeared in the work of M. Kapranov and V. Schechtman (\cite{schobers}).

They introduced the category $\mathcal{A}_2$ which is equivalent to $\mathcal{A}_1$ (see \cite{KS}). It is defined as a category of triplets of vector spaces together with some linear transformations between them. Once more, one can replace vector spaces and linear transformations by categories and functors, and get the notion of a spherical pair. Its definition is purely triangulated. Given DG enhancements, one can produce two spherical functors from a spherical pair, adjoint to each other.

There exists another notion related to spherical functors, the so-called 4-periodic semiorthogonal decomposition (see \cite{Bodz}). It consists of four subcategories forming appropriate semiorthogonal decompositions. Once more, this definition does not use DG categories and such 4-periodic semiorthogonal decomposition yields a spherical pair.

D. Halpern-Leistner and I. Schipman proved in \cite{HLS} that any spherical functor yields a 4-periodic semiorthogonal decomposition. The ambient category was constructed via DG gluing along a DG functor.  A natural question whether the notions of spherical pair and 4-periodic semiorthogonal decomposition are equivalent in the realm of triangulated categories remains open. The theorem proved in this paper suggests that the answer to this question is positive.

We introduce the category $\mathcal{C}$ related to the notion of a 4-periodic semiorthogonal decomposition in a similar manner as categories $\mathcal{A}_1, \mathcal{A}_2$ are connected to spherical functors and spherical pairs. It is a category of vector spaces together with four subspaces whose appropriate direct sums are isomorphic to the original space. We prove the following theorem.

\begin{theorem}
	The category $\mathcal{C}$ is equivalent to the category $\mathsf{Perv}(D)$.
\end{theorem}

Our proof is based on the equivalences mentioned above, in fact we construct quasi-inverse equivalences $S : \mathcal{C} \xrightarrow{\simeq} \mathcal{A}_2$ and $T : \mathcal{A}_2 \xrightarrow{\simeq}\mathcal{C}$.

\section*{Acknowledgements}
	This paper is based on the author's master thesis written under the supervision of Agnieszka Bodzenta. The author was partially supported by the Polish National Science Centre Grant No. 2018/31/D/STI/03375.
	
\section{Preliminaries}
	Throughout this paper we are working over the fixed field $\mathbf{k}$. 
	
	Let $D$ be a two dimensional real disc. We fix a stratification of $D$ by the origin $0$ and its complement $U = D \setminus \{ 0 \}$, and we fix the middle perversity, i.e. a function $p : \{ 0, U \} \to \mathbf{Z}$ such that $p(0) = 0, \ p(U) = -1$. Let $\mathcal{D}(D)$ be the derived category of the category of sheaves of $\mathbf{k}$-vector spaces on $D$ and denote by $i_{0}, i_{U}$ the inclusions of both strata. Recall that a complex of sheaves is called \emph{constructible} with respect to the given stratification if its cohomology sheaves are locally constant after the restriction to each stratum.
	 
	\begin{definition}
		The category $\mathsf{Perv}(D)$ of perverse sheaves on $D$ with respect to the above stratification and the perversity function consists of objects $\mathcal{F} \in \mathcal{D}(D)$ such that:
		\begin{enumerate}
			\item $\mathcal{F}$ is constructible,
			\item $\mathcal{H}^n(i_0^*(\mathcal{F})) = 0 \text{ for } n > 0$,
			\item $\mathcal{H}^n(i_{0}^!(\mathcal{F})) = 0 \text{ for } n < 0$,
			\item $\mathcal{H}^n(i_{U}^*(\mathcal{F})) = 0 \text{ for } n > -1$,
			\item $\mathcal{H}^n(i_{U}^!(\mathcal{F})) = 0 \text{ for } n < -1$.
		\end{enumerate}
	\end{definition}

	The general definition of the category of perverse sheaves on an arbitrary stratified topological space can be found in \cite{faisceaux}.
	
	Let $\mathsf{Vect}$ be the category of $\mathbf{k}$-vector spaces. We define categories $\mathcal{A}_2$ and $\mathcal{C}$ as:
	
	\begin{definition}
		Objects of the category $\mathcal{A}_2$ are:
		\begin{equation*}
		\begin{aligned}
		&Obj(\mathcal{A}_2) \coloneqq &&\{
		\begin{tikzcd} [ampersand replacement=\&]
		{E_-} \& {E_0} \& {E_+}
		\arrow["{\delta_-}"', shift right=2, from=1-1, to=1-2]
		\arrow["{\gamma_-}"', shift right=2, from=1-2, to=1-1]
		\arrow["{\gamma_+}"', shift right=2, from=1-2, to=1-3]
		\arrow["{\delta_+}"', shift right=2, from=1-3, to=1-2]
		\end{tikzcd}
		: E_i \in \mathsf{Vect}, \ \gamma_{+}\delta_{+} = 1_{E_+} , \  \gamma_{-}\delta_{-} = 1_{E_-} , \\
		&&&\gamma_{-}\delta_{+} \ \text{is an isomorphism} , \gamma_{+}\delta_{-} \ \text{is an isomorphism} \}.
		\end{aligned}
		\end{equation*}
		We will often denote such an object as $(E_-, E_0, E_+)$.
		
		Morphisms in $\mathcal{A}_2$ are triplets of linear transformations such that all squares in the following diagram commute
		\begin{equation} \label{mor2}
		\begin{tikzcd} 
		{E_-} & {E_0} & {E_+} \\ \\
		{F_-} & {F_0} & {F_+}
		\arrow["{\delta_-}"', shift right=2, from=1-1, to=1-2]
		\arrow["{\gamma_+}"', shift right=2, from=1-2, to=1-3]
		\arrow["{\eta_-}"', shift right=2, from=3-1, to=3-2]
		\arrow["{\xi_+}"', shift right=2, from=3-2, to=3-3]
		\arrow["{\delta_+}"', shift right=2, from=1-3, to=1-2]
		\arrow["{\gamma_-}"', shift right=2, from=1-2, to=1-1]
		\arrow["{\eta_+}"', shift right=2, from=3-3, to=3-2]
		\arrow["{\xi_-}"', shift right=2, from=3-2, to=3-1]
		\arrow["{e_-}"', from=1-1, to=3-1]
		\arrow["{e_0}"', from=1-2, to=3-2]
		\arrow["{e_+}", from=1-3, to=3-3]
		\end{tikzcd}
		\end{equation}
		(such a morphism will be denoted as $(e_-, e_0, e_+)$).
	\end{definition}
	
	\begin{definition} 
		Objects of the category $\mathcal{C}$ are:
		\[
		Obj(\mathcal{C}) \coloneqq \{ (V, A_1 , A_2, B_1, B_2) : \ V \in \mathsf{Vect}, \ A_i, B_i \subseteq V , \ A_i \oplus B_j \simeq V \text{ for all } i, j = 1,2 \}.
		\]
		Such an object will be denoted as $V_{A, B}$.

		Morphisms in $\mathcal{C}$ are linear transformations that preserve the structure:
		\begin{equation*}
		\begin{aligned}
		Hom_{\mathcal{C}}(V_{A, B}, W_{X, Y}) = \{ \varphi \in Hom_{\mathsf{Vect}}(V, W) :
		\varphi(A_i) \subseteq X_i,	\varphi(B_i) \subseteq Y_i \}.	
		\end{aligned}
		\end{equation*}
	\end{definition}
	
	\begin{proposition} \cite{KS}
		The category $\mathcal{A}_2$ is equivalent to the category $\mathsf{Perv}(D)$.
	\end{proposition}
	
\section{Equivalence}
	We define two functors between $\mathcal{A}_2$ and $\mathcal{C}$ and prove that they are quasi-inverse equivalences.
	\begin{definition}
		Let us define a functor $S: \mathcal{C} \to \mathcal{A}_2$ on objects as:
		\[
		S(V_{A, B}) = \begin{tikzcd} [ampersand replacement=\&]
		{A_1} \& {V} \& {A_2}
		\arrow["{i_1}"', shift right=2, from=1-1, to=1-2]
		\arrow["{\pi_1}"', shift right=2, from=1-2, to=1-1]
		\arrow["{\pi_2}"', shift right=2, from=1-2, to=1-3]
		\arrow["{i_2}"', shift right=2, from=1-3, to=1-2]
		\end{tikzcd}
		\]
		where $i_1, i_2$ denote the inclusions of subspaces and $\pi_k$ are the projections $V \simeq A_k \oplus B_k \to A_k$ with kernels $B_k$.
		
		Let $W_{X, Y}$ be another object in $\mathcal{C}$ with morphisms $j_k, \rho_k$ such as $i_k, \pi_k$ above. Let $\varphi: V_{A, B} \to W_{X, Y}$ be a morphism in $\mathcal{C}$. We define $S(\varphi)$ as:
		\begin{equation} \label{Smor}
		\begin{tikzcd} 
		{A_1} & {V} & {A_2} \\ \\
		{X_1} & {W} & {X_2}
		\arrow["{i_1}"', shift right=2, from=1-1, to=1-2]
		\arrow["{\pi_2}"', shift right=2, from=1-2, to=1-3]
		\arrow["{j_1}"', shift right=2, from=3-1, to=3-2]
		\arrow["{\rho_2}"', shift right=2, from=3-2, to=3-3]
		\arrow["{i_2}"', shift right=2, from=1-3, to=1-2]
		\arrow["{\pi_1}"', shift right=2, from=1-2, to=1-1]
		\arrow["{j_2}"', shift right=2, from=3-3, to=3-2]
		\arrow["{\rho_1}"', shift right=2, from=3-2, to=3-1]
		\arrow["{\varphi \circ i_1}"', from=1-1, to=3-1]
		\arrow["{\varphi}"', from=1-2, to=3-2]
		\arrow["{\varphi \circ i_2}", from=1-3, to=3-3]
		\end{tikzcd}
		\end{equation}
	\end{definition}

	\begin{lemma}
		The functor $S$ is well-defined.
	\end{lemma}
	\begin{proof}
		To prove that $S$ is well-defined on objects, the only non-obvious calculations are whether $\pi_2 i_1$ and $\pi_1 i_2$ are isomorphisms. We will show that they are injective and surjective.
	
		Let us focus on $\pi_2 i_1$. Assume that $\pi_2 i_1 (x) = 0$ for some $x \in A_1$. By definition of $\pi_2$, this means that $i_1(x) \in B_2$. Since $V \simeq A_1 \oplus B_2$, an element $x \in A_1 \cap B_2$ has to be zero.
		Now, take any $y \in A_2$. Since $V \simeq A_1 \oplus B_2$, we can write $y = a_1 + b_2$ for unique $a_1 \in A_1, b_2 \in B_2$. Then, as $b_2 \in ker(\pi_2)$ and $\pi_2i_2 = 1$, we have $\pi_2 i_1(a_1) = \pi_2 (y - b_2) = y$. The proof for $\pi_1 i_2$ is analogous. Thus, $S$ is well-defined on objects.
		
		It remains to show that $S(\varphi)$ is indeed a morphism in $\mathcal{A}_2$. First of all, the left and right arrows are well-defined because $\varphi$ is a morphism in $\mathcal{C}$. Hence, we need to show that both squares in the diagram (\ref{Smor}) commute.
	
		Since $i_k, j_k$ are inclusions, we see that $j_k \varphi i_k = \varphi i_k$.
	
		Now, denote by $b_1$ and $y_1$ inclusions $B_1 \to A_1 \oplus B_1$ and $Y_1 \to X_1 \oplus Y_1$ and by $\beta_1, \gamma_1$ the projections $A_1 \oplus B_1 \to B_1$ and $X_1 \oplus Y_1 \to Y_1$. Then:
		\[
		\rho_1 \varphi = \rho_1 \varphi (i_1 \pi_1 + b_1 \beta_1) = \rho_1 \varphi i_1 \pi_1 + \rho_1 \varphi b_1 \beta_1 = \rho_1 j_1 \varphi i_1 \pi_1 + \rho_1 y_1 \varphi b_1 \beta_1 = \varphi i_1 \pi_1.
		\]
		Analogously, one shows that $\rho_2 \varphi = \varphi i_2 \pi_2$.
	\end{proof}

	\begin{definition}
		Let us define a functor $T: \mathcal{A}_2 \to \mathcal{C}$ on objects as:
		\[
		T(E_-, E_0, E_+) = (E_0, \delta_-(E_-), \delta_+(E_+), ker(\gamma_-), ker(\gamma_+)). 
		\]
	
		If $(e_-, e_0, e_+)$ is a morphism in $\mathcal{A}_2$ (as in (\ref{mor2})), we put $T(e_-, e_0, e_+) = e_0$.
	\end{definition}

	\begin{lemma}
		Functor $T$ is well-defined.
	\end{lemma}
	\begin{proof}
		To check that $T$ is well-defined on objects, we need to prove the appropriate decompositions of $E_0$ as direct sums. The computations that $E_0 \simeq \delta_-(E_-) \oplus ker(\gamma_-) \simeq \delta_+(E_+) \oplus ker(\gamma_+)$ are straightforward from the definition of $\mathcal{A}_2$.
		
		We show that $E_0 \simeq \delta_+(E_+) \oplus ker(\gamma_-)$. The proof for $\delta_-(E_-) \oplus ker(\gamma_+)$ is analogous. Let $x \in \delta_+(E_+) \cap ker(\gamma_-)$. Then $x = \delta_+(x')$ for some $x' \in E_+$. The condition $0 = \gamma_-(x) = \gamma_-(\delta_+(x'))$ implies that $x' = 0$ as $\gamma_-\delta_+$ is an isomorphism. Hence $x = 0$.
		
		Take any $v \in E_0$. We already know that we can express $v = \delta_-(v_-) + k_-$ with $v_- \in E_-, k_- \in ker(\gamma_-)$. Since $\gamma_-\delta_+$ is an isomorphism, we can find $v_+ \in E_+$ such that $v_- = \gamma_-\delta_+(v_+)$, i.e. $v = \delta_-(\gamma_- \delta_+ (v_+)) + k_-$.
		
		Denote by $i, \pi$ the inclusion and projection for $ker(\gamma_-)$ in the decomposition $E_0 \simeq \delta_-(E_-) \oplus ker(\gamma_-)$. Then we get $1 = \delta_-\gamma_- + i \pi$, i.e. $\delta_-\gamma_- = 1 - i\pi$ and
		\[
		v = (1 - i\pi)(\delta_+(v_+)) + k_- = \delta_+(v_+) - i\pi\delta_+(v_+) + k_-.
		\]
		The first summand lays in $\delta_+(E_+)$ and both second and third belong to $ker(\gamma_-)$. Hence, $E_0 = \delta_+(E_+) \oplus ker(\gamma_-)$.
		
		This proves that $T$ is well-defined on objects.
		
		It remains to show that $T(e_-, e_0, e_+) = e_0$ is a morphism in $\mathcal{C}$, i.e. that images of $\delta_-(E_-)$ and $ker(\gamma_-)$ are contained in appropriate subspaces (the images of $\delta_+(E_+)$ and $ker(\gamma_+)$ are proved analogously).
		
		From the definition of $(e_-, e_0, e_+)$ (see (\ref{mor2}) for notations) we get that
		\[
		e_0 \delta_- = \eta_- e_-,
		\]
		\[
		\xi_- e_0 = e_- \gamma_-.
		\]
		In particular, the second condition gives us immediately $e_0(ker(\gamma_-)) \subseteq ker(\xi_-)$.
		Let us compute:
		\[
		e_0(\delta_-(E_-)) = \eta_- e_-(E_-) \subseteq \eta_-(F_-).
		\]
		
		This shows that $T(e_-, e_0, e_+)$ is indeed a morphism in $\mathcal{C}$ and finishes the proof.
	\end{proof}
	
	We have two well-defined functors and we can state the main theorem.
	
	\begin{theorem}
		Functors $S$ and $T$ are quasi-inverse equivalences of categories.
	\end{theorem}
	\begin{proof}
		Firstly, we show that $TS(V_{A, B}) = V_{A, B}$:
		\[
		TS(V_{A, B}) =  T(\begin{tikzcd} [ampersand replacement=\&]
		{A_1} \& {V} \& {A_2}
		\arrow["{i_1}"', shift right=2, from=1-1, to=1-2]
		\arrow["{\pi_1}"', shift right=2, from=1-2, to=1-1]
		\arrow["{\pi_2}"', shift right=2, from=1-2, to=1-3]
		\arrow["{i_2}"', shift right=2, from=1-3, to=1-2]
		\end{tikzcd} ) = (V, i_1(A_1), i_2(A_2), ker(\pi_1), ker(\pi_2)) = (V, A_1, A_2, B_1, B_2),
		\]
		where the last equality follows from the definitions of $i_1, i_2, \pi_1$ and $\pi_2$.
		
		Hence, $TS$ is the identity functor, since the equality on morphisms is straightforward.
		
		Secondly, we show that $ST$ is naturally isomorphic to the identity functor.
		\[
		ST(E_-, E_0, E_+) = S(E_0, \delta_-(E_-), \delta_+(E_+), ker(\gamma_-), ker(\gamma_+)) = \begin{tikzcd} [ampersand replacement=\&]
		{\delta_-(E_-)} \& {E_0} \& {\delta_+(E_+)}
		\arrow["{i_-}"', shift right=2, from=1-1, to=1-2]
		\arrow["{\delta_-\gamma_-}"', shift right=2, from=1-2, to=1-1]
		\arrow["{\delta_+\gamma_+}"', shift right=2, from=1-2, to=1-3]
		\arrow["{i_+}"', shift right=2, from=1-3, to=1-2]
		\end{tikzcd} 
		\]
		where $i_-, i_+$ are subspace inclusions.
		
		We define a natural transformation $M: Id \to ST$. For $(E_-, E_0, E_+) \in \mathcal{A}_2$, $M(E_-, E_0, E_+)$ is given by:
		\[
		\begin{tikzcd} 
		{E_-} & {E_0} & {E_+} \\ \\
		{\delta_-(E_-)} & {E_0} & {\delta_+(E_+)}
		\arrow["{\delta_-}"', shift right=2, from=1-1, to=1-2]
		\arrow["{\gamma_+}"', shift right=2, from=1-2, to=1-3]
		\arrow["{i_-}"', shift right=2, from=3-1, to=3-2]
		\arrow["{\delta_+\gamma_+}"', shift right=2, from=3-2, to=3-3]
		\arrow["{\delta_+}"', shift right=2, from=1-3, to=1-2]
		\arrow["{\gamma_-}"', shift right=2, from=1-2, to=1-1]
		\arrow["{i_+}"', shift right=2, from=3-3, to=3-2]
		\arrow["{\delta_-\gamma_-}"', shift right=2, from=3-2, to=3-1]
		\arrow["{\delta_-}"', from=1-1, to=3-1]
		\arrow["{1}"', from=1-2, to=3-2]
		\arrow["{\delta_+}", from=1-3, to=3-3]
		\end{tikzcd}
		\]	
		It is obviously a morphism in $\mathcal{A}_2$. To show that $M$ is indeed a natural transformation we look at the following diagram:
		
		\[\begin{tikzcd}
		{E_-} & {E_0} & {E_+} && {\delta_-(E_-)} & {E_0} & {\delta_+(E_+)} \\
		\\
		{F_-} & {F_0} & {F_+} && {\eta_-(F_-)} & {F_0} & {\eta_+(F_+)}
		\arrow["{\delta_-}"', shift right=2, from=1-1, to=1-2]
		\arrow["{\gamma_-}"', shift right=2, from=1-2, to=1-1]
		\arrow["{\gamma_+}"', shift right=2, from=1-2, to=1-3]
		\arrow["{\delta_+}"', shift right=2, from=1-3, to=1-2]
		\arrow["{e_-}", from=1-1, to=3-1]
		\arrow["{e_0}", from=1-2, to=3-2]
		\arrow["{e_+}", from=1-3, to=3-3]
		\arrow["{M(E_-, E_0, E_+)}", from=1-3, to=1-5]
		\arrow["{M(F_-, F_0, F_+)}", from=3-3, to=3-5]
		\arrow["{i_-}"', shift right=2, from=1-5, to=1-6]
		\arrow["{\delta_-\gamma_-}"', shift right=2, from=1-6, to=1-5]
		\arrow["{j_-}"', shift right=2, from=3-5, to=3-6]
		\arrow["{\eta_-\xi_-}"', shift right=2, from=3-6, to=3-5]
		\arrow["{e_0 \circ i_+}", from=1-7, to=3-7]
		\arrow["{e_0 \circ i_-}", from=1-5, to=3-5]
		\arrow["{\xi_+}"', shift right=2, from=3-2, to=3-3]
		\arrow["{\eta_+}"', shift right=2, from=3-3, to=3-2]
		\arrow["{\eta_-}"', shift right=2, from=3-1, to=3-2]
		\arrow["{\xi_-}"', shift right=2, from=3-2, to=3-1]
		\arrow["{e_0}", from=1-6, to=3-6]
		\arrow["{\delta_+\gamma_+}"', shift right=2, from=1-6, to=1-7]
		\arrow["{i_+}"', shift right=2, from=1-7, to=1-6]
		\arrow["{\eta_+\xi_+}"', shift right=2, from=3-6, to=3-7]
		\arrow["{j_+}"', shift right=2, from=3-7, to=3-6]
		\end{tikzcd}\]
		
		It is commutative since:
		\begin{equation*}
		\begin{aligned}
		&e_0 i_- \delta_- = e_0 \delta_- = \eta_- e_-, \\
		&e_0 \ 1 = 1 \ e_0, \\
		&e_0 i_+ \delta_+ = e_0 \delta_+ = \eta_+ e_+.
		\end{aligned}
		\end{equation*}
		
		It remains to show that $M(E_-, E_0, E_+)$ is an isomorphism in $\mathcal{A}_2$.
		
		We claim that the following diagram defines its inverse $M(E_-, E_0, E_+)^{-1}$:
		\[
		\begin{tikzcd} 
		{E_-} & {E_0} & {E_+} \\ \\
		{\delta_-(E_-)} & {E_0} & {\delta_+(E_+)}
		\arrow["{\delta_-}"', shift right=2, from=1-1, to=1-2]
		\arrow["{\gamma_+}"', shift right=2, from=1-2, to=1-3]
		\arrow["{i_-}"', shift right=2, from=3-1, to=3-2]
		\arrow["{\delta_+\gamma_+}"', shift right=2, from=3-2, to=3-3]
		\arrow["{\delta_+}"', shift right=2, from=1-3, to=1-2]
		\arrow["{\gamma_-}"', shift right=2, from=1-2, to=1-1]
		\arrow["{i_+}"', shift right=2, from=3-3, to=3-2]
		\arrow["{\delta_-\gamma_-}"', shift right=2, from=3-2, to=3-1]
		\arrow["{\gamma_-}"', from=3-1, to=1-1]
		\arrow["{1}"', from=3-2, to=1-2]
		\arrow["{\gamma_+}", from=3-3, to=1-3]
		\end{tikzcd}
		\]	
		To check that it defines a morphism we compute:
		\[
		\gamma_-\delta_-\gamma_- = \gamma_-,
		\]
		\[
		\gamma_+\delta_+\gamma_+ = \gamma_+.
		\]
		Now, for $x \in E_-$:
		\[
		\delta_-\gamma_-(\delta_-(x)) = \delta_-(x) = i_-\delta_-(x)
		\]
		and for $x' \in E_+$:
		\[
		\delta_+\gamma_+(\delta_+(x')) = \delta_+(x) = i_+\delta_-(x').
		\]
		It remains to show that both compositions give us identities.
		
		$M(E)^{-1}M(E) = 1$ because it is the identity on $E_-, E_0$ and $E_+$.
		
		$M(E)M(E)^{-1}$ gives identity on $E_0$ and $\delta_i\gamma_i : \delta_i(E_i) \to \delta_i(E_i)$.
		It is an identity on $\delta_i(E_i)$ since $\delta_i\gamma_i(\delta_i(x)) = \delta_i(x)$.
		
		Thus, $M(E)$ is an isomorphism, $M$ is a natural isomorphism and $S$ and $T$ are indeed quasi-inverse equivalences of categories.
	\end{proof}

\bibliography{PSBibliography}

\end{document}